\documentclass[11pt,reqno]{amsart}
\usepackage[T1]{fontenc}
\usepackage{lmodern,amsmath,amssymb,amsthm,mathtools,microtype}
\usepackage[margin=1.05in]{geometry}
\usepackage{enumitem,xcolor}
\definecolor{linkblue}{RGB}{20,55,105}
\usepackage[colorlinks=true,linkcolor=linkblue,citecolor=linkblue,urlcolor=linkblue]{hyperref}
\newtheorem{theorem}{Theorem}[section]
\newtheorem{proposition}[theorem]{Proposition}
\newtheorem{lemma}[theorem]{Lemma}
\theoremstyle{remark}
\newtheorem{remark}[theorem]{Remark}
\numberwithin{equation}{section}
\newcommand{\R}{\mathbb R}

\newcommand{\spt}{\operatorname{spt}}
\newcommand{\norm}[1]{\lVert#1\rVert}
\title{Stationary varifolds with singularities II}
\author[C. De Lellis]{Camillo De Lellis}
\author[J. Hirsch]{Jonas Hirsch}
\author[Z. Lihn]{Zachary Lihn}
\author[L. Spolaor]{Luca Spolaor}
\date{}

\begin{document}

\begin{abstract}
For every dimension $m\geq 3$, we show the existence of an open set $U\subset \mathbb R^{m+1}$, a smooth Riemannian metric $g$ on it, and an integral $m$-dimensional stationary varifold $V$ in $(U,g)$ with the following properties. $V$ has a flat tangent plane of multiplicity $2$ at an interior point $p$, it is a smooth immersed minimal surface in $U\setminus \{p\}$, and it has infinite topology in any neighborhood of $p$. The construction starts from an idea of three of the authors, who in \cite{DHS} tried to build a similar example with $C^{m-1, \alpha}$ regularity for $g$. That previous attempt contained an important error, which has been overcome using a crucial suggestion of gpt-astra-6. 
\end{abstract}

\maketitle

\section{Introduction}

In this note we report a proof of the following theorem.

\begin{theorem}\label{thm:main}
For every $m\geq3$ there are a bounded open neighborhood
$U\subset\R^{m+1}$ of $0$, a smooth metric $g$ on $U$, and a stationary integral varifold $V$ in the Riemannian manifold $(U,g)$ with the following properties. 
\begin{itemize}
\item[(a)] In $U\setminus \{0\}$ the support of the varifold is a properly immersed smooth minimal surface $\Sigma$ and the multiplicity of the varifold is $1$.
\item[(b)] The varifold tangent to $V$ at $0$ is an $m$-dimensional plane with multiplicity $2$.
\item[(c)] The topology of $\Sigma$ is infinite in any neighborhood of zero.
\end{itemize}
\end{theorem}

In a first version of \cite{DHS} three of the authors claimed the same theorem with a metric of class $C^{m-1, \alpha}$ for $\alpha$ positive and a similar result for stationary varifolds which enjoy, additionally, stability of the regular part, but Gerrard Orriols pointed out a mistake in the argument. While the strategy did not break down completely, it reduced the regularity of the metric to the rather unsatisfactory $C^{1,1-\varepsilon}$ class. Further attempts to fix the problem were unsuccessful. In early September 2026 we prompted astra-gpt-6 for a suggestion on how to circumvent the stumbling block, and the model suggested one useful trick which we had overlooked in our previous attempts, together with a second simplification of our arguments which improves the regularity to $C^\infty$. However, this second simplification does not allow us to prove an analogous statement for varifolds with stable regular part. We instead believe that another approach, due to us, would allow to extend the theorem to varifolds with stable regular part, but the metric would stay in the regularity class $C^{m-1,\alpha}$, cf. Remarks \ref{r:embedded}, \ref{r:embedded-2}, \ref{r:embedded-3}, and \ref{r:embedded-4}. We will explain this in detail, in particular singling out the suggestions of astra-gpt-6. 

We further prompted astra-gpt-6 to give a similar example in dimension $m=2$ and this also proved to be possible (with the difference that the tangent plane at the singular point has multiplicity $3$, cf. \cite{DHLS2}). While, from a general viewpoint, the strategy adopted in \cite{DHLS2} is still similar to the one of this paper, the two additional new ingredients due to astra-gpt-6 are in fact rather far from what we tried and we have therefore decided to report the result and the proof in a separate expository article.  

Before coming to the proof of Theorem \ref{thm:main}, we briefly mention our initial motivation. As explained in the introduction of \cite{DHS}, after the celebrated result of Allard \cite{All} that the interior singular set of stationary integral varifolds is meager, the question of its optimal dimension and structure, without stronger variational assumptions, has remained virtually untouched and has attracted a lot of attention in the last 50 years, see for instance \cite{BMW,Brakke, BDF,  CCSvarifolds,DelellisAllard, DPGS,HirschSpolaor2024, HS, MenneJGA, SavinAllard,SW}. A clear outcome of these works is that the main obstruction to further progress is the possible presence of topology on the smooth part of the varifold accumulating to flat singularities. Theorem \ref{thm:main} and its counterpart in \cite{DHLS2} show that this obstruction can in fact occur, at least in ambient $C^\infty$ metrics and if the varifold is smooth but immersed away from the singularity. The two interesting remaining open questions are whether it is possible to construct similar examples in the Euclidean space (or with a real-analytic metric) and make them embedded outside the singular point. While these problems seem to be at a different level in terms of difficulty, given the astonishing recent AI developments, any such evaluation should be taken with a grain of salt.

\section{Outline of the construction}\label{s:section}

The strategy proposed in \cite{DHS} is to start from a varifold $V_0$ which is the union of two smooth minimal graphs in a cylinder of the Euclidean space, both containing the origin and having the same tangent at it. In order to simplify our discussion we can assume the tangent to be horizontal and the two graphs to be the graphs of the functions $u: B_1 \to \mathbb R$ and $-u$. This specific choice will become more important later on. 

We next wish to take a sequence of points $p_k\in B_1\setminus \{0\}$ converging to the origin (for instance points $p_k$ where $u (p_k)$ achieves its maximum over $\partial B_{|p_k|}$) and some appropriately chosen (sufficiently small) radii $r_k\downarrow 0$. Consider the two disks formed by the graphs of $u$ and $-u$ above $B_{r_k} (p_k)\times \{0\}$, in other words ${\rm spt}\, (V_0) \cap (B_{r_k} (p_k)\times \mathbb R)$. Our goal is to replace them with a single connected surface, a ``cylindrical neck'', and then change the metric so that this new surface is still minimal. The natural candidate to use is the classical catenoid and the glueing can be done using a partition of unity on the graphical ends of the catenoidal neck. Adjusting the metric so that the new surface is minimal requires changing it in cylindrical regions of the form $(B_{r_k} (p_k) \setminus \overline{B}_{c_0 r_k} (p_k))\times \mathbb R$ and we can then do it over all such regions separately. In the remaining portion of the space the surface is already a Euclidean minimal surface and the metric stays Euclidean. Therefore the new surface (which has infinite topology) will be globally minimal in the new metric. Since clearly this procedure retains the $C^\infty$ regularity of the metric in every compact set $K\subset \mathbb R^{m+1}\setminus \{0\}$, the crucial point is how regular the metric can be made at the origin. 

\medskip

A key observation of \cite{DHS} is that, in dimension $m\geq 3$, this strategy promises a much better answer than in dimension $m=2$. The reason is that, for $m\geq 3$ any catenoid decays at infinity towards two parallel planes at a rate $O (|x|^{-m+2})$, while for $m=2$ catenoids diverge logarithmically. However there is, unfortunately, no good choice of the radii $r_k$ in which ${\rm spt}\, (V_0) \cap (B_{r_k} (p_k)\times \mathbb R)$ is ``close enough'' to the two planar ends of a catenoid. The main reason is that, compared to the height $u (p_k)$, the angle formed by $\nabla u (p_k)$ with $0$ is always too large: this is a consequence of the real analyticity of $u$. 

One possible attempt to fix this issue is to look for a better model to glue in: an ideal one is in particular a connected minimal surface whose boundary is given by the spheres forming ${\rm spt}\, (V_0) \cap (\partial B_{r_k} (p_k) \times \mathbb R)$ . Classical PDE estimates would imply a good graphical region very close to the graphs of $u$ and $-u$, and indeed if such a model existed the regularity of the metric could be pushed much further. However the obstruction to find the existence of the ideal model is exactly that the two boundary spheres are not ``parallel enough'' and in particular this idea does not address the main issue explained above. 

\medskip

After prompting astra-gpt-6 for a way to fix our strategy, the model suggested the following workaround. Split the variables $x\in \mathbb R^m$ as $(x_1, x_2, \bar{x})$. Then select a function $u(x)$ of the form 
\begin{equation}\label{e:Ansatz}
u (x) = w (x_1, x_2) + f (\bar{x})\, .
\end{equation}
$w$ is chosen to be a classical minimal graph in two space dimensions, with $w(0)=0$, $\nabla w (0)=0$, and $D^2 w (0) \neq 0$. $f$ is chosen to be an arbitrary smooth function with small $C^3$ norm, which is everywhere positive except at the origin. It is not difficult to see that this allows to perturb the Euclidean metric to a smooth $g_0$ in which the graph of $u$ is minimal. More importantly this new metric $g_0$ can be chosen so that:
\begin{itemize}
\item Both the graphs of $u$ and $-u$ are minimal;
\item if $f$ is constant on $\Omega \subset \mathbb R^{m-2}$, then $g_0$ is Euclidean on the open region $\mathbb R^2 \times \Omega\times \mathbb R$.
\end{itemize}
Since $f$ is otherwise arbitrary, we can now choose it so that there are a sequence of points $\bar{p}_k\in \mathbb R^{m-2}$ converging to $0$ and a sequence of radii $r_k\downarrow 0$ with the property that $f$ is constant over $B_{r_k} (\bar p_k)$. Consider now $p_k = (0, \bar p_k)$. On $B_{r_k} (p_k) \times \mathbb R$ the two functions $u$ and $-u$ are given by $w (x_1, x_2) + f (\bar{p}_k)$ and $- w (x_1, x_2) - f (\bar{p}_k)$. Therefore, since $\nabla w (0) = 0$, we are in the ideal position that $\nabla u (p_k) =0$ and $u (p_k) >0$. Tuning the three available parameters correctly ($|p_k|$, $r_k$, and $f (\bar p_k)$)), it is then possible to find a connected minimal surface with boundary ${\spt}\, (V_0) \cap (\partial B_{r_k} (0, \bar{p}_k))$ (in order to prove this statement it is however convenient to assume some symmetry for the function $w$). In particular the program outlined above can then be completed successfully. 

There is however a second important trick suggested by astra-gpt-6. In our initial approach we insist on keeping the catenoidal models to be glued in {\em embedded}. This is possible, but the argument would only deliver a $C^{m-1,\alpha}$ metric. If one sacrifices the embeddedness of the ``local catenoidal model'' to be glued in, then it is possible to improve the regularity of the metric to $C^\infty$, cf. Remarks \ref{r:embedded},\ref{r:embedded-2}, \ref{r:embedded-3}, and \ref{r:embedded-4}. A 
specific simple way of perturbing the ambient metric, which leverages the symmetry of the construction, allows the implementation of the latter strategy. Our original idea, which is more sophisticated at the PDE level, is likely to also deliver a statement as in Theorem \ref{thm:main} for varifolds with stable regular part, albeit with a $C^{m-1,\alpha}$ metric. However we do not attempt this here.  

\medskip

In the rest of the note we will give the details of the proof. We will keep the notation $x= (x_1, x_2, \bar{x})$ and keep using $u$, $w$, and $f$ for the three functions outlined above. The varifold $V_0$ is the multiplicity-one varifold supported in the union of the graphs of $u$ and $-u$. We will call it base model surface. The next section will prove the existence of a metric $g_0$ in which the base model surface is minimal and has the property described above. Section \ref{s:catenoids} will make some further assumptions on $w$ and explain how to construct the catenoidal necks. The final section will explain how to choose the parameters and completes the proof of Theorem \ref{thm:main}. 

\section{The base model}

The main point of the construction of the base model is the following proposition.

\begin{proposition}\label{p:base-model}
Let $w: \mathbb R^2 \supset W \to \mathbb R$ be a smooth solution of the (Euclidean) minimal surface equation with $D^2 w (0)\neq 0$ and even in both variables. Let $f: \mathbb R^{m-2} \supset \bar{W} \to \mathbb R$ be a smooth function with $\nabla f (0)=0$. If $\|f\|_{C^2}$ is sufficiently small, then there is a neighborhood $U = W' \times \bar{W}'\times \mathbb R$ of the origin in $\mathbb R^{m+1}$ and a smooth metric $g_0$ on $U$ with the following properties:
\begin{itemize}
\item[(i)] Setting
\[
u (x) = w (x_1, x_2) + f (\bar{x})
\]
the graphs of both $u$ and $-u$ are minimal in $(U, g_0)$.
\item[(ii)] If $\Omega \subset \mathbb R^{m-2}$ is open and $f$ is constant on $\Omega$, then $g_0$ is the Euclidean metric $e$ on the cylinder $W' \times \Omega \times \mathbb R$.
\end{itemize}
\end{proposition}

\begin{remark}
The assumption that $w$ is even in both variables is not really needed. On the other hand it simplifies the proof considerably and the proposition will be in fact used with a model $w$ which satisfies it.     
\end{remark}

A main simple ingredient of the proof, which will be used further in the note, is the following way of finding a metric $g$ which makes the graph of a function $v$ minimal once we know that $v$ solves an elliptic PDE. While this is a variant of other arguments we came up with in our attempts, its specific form has been suggested by gpt-astra-6 and it has the advantage of giving a single metric which makes simultaneously the graphs of $v$ and $-v$ minimal, even when the latter intersects.

\begin{lemma}\label{lem:metric-formula}
Let $v$ be smooth on an open subset $\Omega$ of $\R^m$ and assume that it solves an elliptic partial differential equation in divergence form with variable coefficients $A_{ij}=A_{ji}$:
\begin{equation}\label{e:elliptic}
{\rm div}\, (A (x) \nabla v (x)) = 0\, .
\end{equation}
Define 
\[
d (x) =\sqrt{1+|\nabla v (x)|^2},\qquad B=d A,
 \qquad c=\left(\frac{(\det B) (1+\nabla v \cdot B \nabla v)}{d^2}\right)^{1/m}.
\]
and the following height-independent metric on $\Omega \times \mathbb R$
\begin{equation}\label{eq:metric-formula}
 g_0 (x, y) = c (x) \left(
 \begin{array}{ll}
 B(x)^{-1} & 0 \\
 0 & 1
 \end{array}
\right)
\end{equation}
The graphs of both $v$ and $-v$ are minimal in $(\Omega \times \mathbb R, g)$.
\end{lemma}

\begin{remark}\label{r:canonical}
Observe that $g_0$ depends smoothly on $A$ and $\nabla v$. Moreover, if $v$ is a solution of the classical minimal surface equation, we can take $A= d^{-1} {\rm Id}$ in \eqref{e:elliptic}. With this choice $B = {\rm Id}$, $c=1$, and thus $g_0$ is precisely the Euclidean metric.
\end{remark}

\begin{proof}
Consider the Lagrangian
\[
 L(x,q)=\frac{c (x)^{m/2}}{\sqrt{\det B (x)}}\sqrt{1+q\cdot B (x) q}\, .
\]
The area of a graph $\zeta$ in the metric $g_0$ is given by
\[
\int L (x, \nabla \zeta (x))\, dx
\]
and thus the minimal surface equation in that metric is the usual Euler-Lagrange equation  
\begin{equation}\label{e:EL}
{\rm div}_x\, D_q L (x, \nabla \zeta (x)) = 0\, .
\end{equation}
The choice of $c$ gives $D_qL(x,\pm \nabla v (x))=\pm B (x) \nabla v (x)/d (x)=\pm A (x) \nabla v (x)$ and hence \eqref{e:EL} reduces to \eqref{e:elliptic} when we insert $\zeta = \pm v$.
\end{proof}

\begin{proof}[Proof of Proposition \ref{p:base-model}]
Observe that, because $w$ is assumed to be even in both variables, $\partial_1 w (0, x_2) =0$ and $\partial_2 w(0, x_1) =0$. In particular $D^2 w (0)$ is diagonal. Moreover $\nabla w(0)=0$ and hence, because $w$ solves the minimal surface equation, $\Delta w (0)=0$.
In particular $D^2 w (0)$ has entries $\lambda >0$ and $-\lambda$. Because of Lemma \ref{lem:metric-formula}, our goal is to find a symmetric matrix $A$ such that \eqref{e:elliptic} holds with $v=u$. Moreover, thanks to Remark \ref{r:canonical}, if we succeed to find such a matrix $A$ so that $A(x) = (1+|\nabla u (x)|^2)^{-1/2} {\rm Id}$ on every open region $\Omega$ where $\nabla f$ vanishes identically, then the resulting metric $g_0$ will have the desired property that it is the Euclidean metric (in its standard form $e$ in the Cartesian coordinates) in $W'\times \Omega \times \mathbb R$. We set $d(x) = \sqrt{1+|\nabla u (x)|^2}$ and look for the matrix $A$ in the form $b_0 (x) e_1\otimes e_1 + d(x)^{-1} {\rm Id}$. The condition that \eqref{e:elliptic} holds is thus equivalent to 
\[
\partial_1 (b_0 \partial_1 u) = \underbrace{-{\rm div}_x (d(x)^{-1} \nabla u (x))}_{=:E(x)}\, .
\]
Recall that $\partial^2_{11} u (0, \bar{x}) = \partial^2_{11} w (0) = \lambda \neq 0$. Hence, for $x_1\neq 0$ and in a neighborhood $W'\times \bar{W}'$ of the origin, the partial derivative $\partial_1 u (x) = \partial_1 w (x_1, x_2)$ does not vanish. Hence, for $x_1\neq 0$ we can simply define  
\[
b_0 (x_1, x_2, \bar{x}) = (\partial_1 w (x_1, x_2))^{-1} \int_0^{x_1} E (t, x_2, \bar{x})\, dt\, . 
\]
Using De L'H\^ospital's Theorem we also see immediately that $b_0$ extends continuously to $x_1=0$ as 
\[
b_0 (0, x_2, \bar{x}) = \frac{E (0,x_2, \bar{x})}{\partial^2_{11} w (0, x_2)}\, .
\]
Higher differentiability follows analogously from differentiating the explicit formula and using again De L'H\^ospital's Theorem. 
In order to ensure ellipticity, namely that $A$ is positive definite, we just need $|b_0 (0)|< 1$, since we are free to choose $W'$ and $\bar{W}'$ smaller. Since $E$ depends only on the first and second derivatives of $f$, this is obviously implied by the smallness condition on $\|f\|_{C^2}$.

Finally, if $f$ is constant over a region $\Omega$, then $\nabla f=0$ and $D^2 f=0$ on it. But then $E(x)= 0$ for any $x=(x_1,x_2, \bar{x}) \in W'\times \Omega$ and thus $b_0 (x)=0$. In particular this shows that $A(x) = d(x)^{-1} {\rm Id}$ and completes the proof.
\end{proof}

\section{Catenoidal necks}\label{s:catenoids}

We detail here the construction of the catenoidal model. This section is subdivided in two parts. In the first part we will construct an exact minimal (Euclidean) connected surface $\Sigma_{a,r}$ whose boundary is the union of the graphs of $a+w$ and $-a-w$ over the spheres $\partial B_r (0)$, under the assumption that the real parameters $a$ and $r$ are appropriately chosen. In the second part we will glue smoothly the graphs of $a+w$ and $-a-w$ and the surface $\Sigma_{a,r}$ in a neighborhood of the cylinder $\partial B_r (0) \times \mathbb R$ and perturb the metric so that the resulting surface $\Gamma_{a,r}$ is minimal. The latter will be called the {\em local model}. The following two propositions give all the details.

\begin{proposition}\label{p:exact}
Let $w: \mathbb R^m \supset B_1 \to \mathbb R$ be a solution of the (Euclidean) minimal surface equation which is even in each coordinate and vanishes at the origin. There are constants
\(r_0>0\), \(C>0\), and \(C_k>0\) (for every integer \(k\geq0\)), with the following property.
For every \(r\in(0,r_0)\) and every sufficiently small
\(a\in(0,a_0(r))\), there is an $m$-dimensional surface $\Sigma_{a,r}$ 
and a constant \(H=H(a,r)\) such that 
\begin{itemize}
\item[(a)] \(0<H\leq Cr^2\).
\item[(b)] $\Sigma_{a,r}$ is the image of a smooth minimal immersion $F_{a,r}: \mathbb S^{m-1} \times [-1,1] \to \mathbb R^{m+1}$;
\item[(c)] The restriction of $F_{a,r}$ to $\mathbb S^{m-1} \times \{\pm 1\}$ is a parametrization of the graphs of \(H+w\)
and \(-H-w\) over \(\partial B_r\).
\item[(d)] In the cylinder \(\{(x,z):a<|x|<r\}\), the immersion consists
of two graphical ends, which are graphs of functions \(\zeta\) and \(-\zeta\).
\item[(e)] For every integer \(k\geq0\),
\begin{equation}\label{eq:target}
 \norm{D^k(\zeta-H-w)}_{C^0(\{r/2<|x|<r\})}
 \leq C_k\left(a^{m-1}r^{2-m-k}
                   +a^{3/2}r^{1/2-k}\right).
\end{equation}
\item[(f)] The area of the surface is bounded by $C r^m$. 
\end{itemize}
\end{proposition}

Proposition 4.1 and Lemma 4.5 are adaptations of standard weighted catenoid arguments.
We include proofs to specify the boundary conditions, normalization, and estimates needed
here. Precise precedents for the nonlinear and linear constructions are cited below

\begin{remark}\label{r:embedded}
Note that, if $D^2 w (0) \neq 0$ and  the height $H (a,r)$ is much smaller than $r^2$, then the constructed surface will not be embedded. The final argument for Theorem \ref{thm:main} forces indeed a choice of $a$ which is smaller than any power of $r$ and the size of $H (a,r)$ is indeed proportional to $a$. In particular the neck regions in the final surface $\Sigma$ of Theorem \ref{thm:main} will indeed not be embedded. 

In our original considerations we have a more sophisticated way of proving Proposition \ref{p:exact} which allows to keep the surface embedded, by imposing $a\geq C r^2$ for $C$ large enough, while leading to the improved estimate
\begin{equation}\label{e:sharper}
 \norm{D^k(\zeta-H-w)}_{C^0(\{r/2<|x|<r\})}
 \leq C_k a^{m-1}r^{2-m-k}\, .
\end{equation}
We will explain in Remark \ref{r:embedded-2} and \ref{r:embedded-4} how this would allow us to prove Theorem \ref{thm:main} with a surface $\Sigma$ which is embedded in the neck regions, but a corresponding metric (making $\Sigma$ a minimal surface) which has only $C^{m-1,\alpha}$ regularity. Moreover, we will outline in Remark \ref{r:embedded-3} how to achieve the sharper estimate \eqref{e:sharper}. 
\end{remark}

The next proposition details the glueing and metric perturbation. We fix a radial cut-off function $\chi$ which is identically $1$ on $B_{5/8}$ and whose support is contained in $B_{7/8}$. We then scale it as $\chi_r (x) := \chi (x/r)$. Finally we fix the function $w$ of Proposition \ref{p:base-model} and define the surface $\Gamma_{a,r}$ as follows
\begin{itemize}
\item $\Gamma_{a,r} = \Sigma_{a,r}$ in the cylinder $B_{r/2}\times \mathbb R$;
\item $\Gamma_{a,r}$ is the graph of the functions $\pm \xi := \pm ( \chi_r \zeta + (1-\chi_r) (w+ H (a,r)))$ on $(B_r\setminus B_{r/2})\times \mathbb R$;
\item $\Gamma_{a,r}$ is the graph of the functions $\pm (w+H(a,r))$ on the rest of the domain where $w$ is defined.
\end{itemize}
The cut-off function is chosen radial so that $\xi$ retains the symmetries of $\zeta$ and $w$. 

\begin{proposition}\label{p:glued}
Let $w$ be as in Proposition \ref{p:base-model}, $a$, and $r$ be as in Proposition \ref{p:exact} and $\Gamma_{a,r}$ be defined as above. Then, provided $r$ is sufficiently small (depending only on $w$) and $a$ is sufficiently small (depending on $w$ and $r$), there is a metric $g_{a,r}$ on $W'\times \mathbb R^{m-2}\times \mathbb R$ such that
\begin{itemize}
\item[(a)] $W'$ is a sufficiently small neighborhood of $0$ in $\mathbb R^2$, chosen independently of the parameters $a$ and $r$;
\item[(b)] $g_{a,r}$ coincides with the Euclidean metric (in its standard form $e$) outside the slab $\{|\bar{x}|< r\}$
\item[(c)] $\Gamma_{a,r}$ is a minimal surface in the metric $g_{a,r}$;
\item[(d)] For every $k$ there is a constant $C_k$ (independent of $a$ and $r$) such that 
\begin{equation}\label{e:estimate-g}
\|g_{a,r}-e\|_{C^k} \leq C_k \left(a^{m-1}r^{-m-k}
                   +a^{3/2}r^{-k-3/2}\right)\, .
\end{equation}
\end{itemize}
\end{proposition}

\begin{remark}\label{r:embedded-2}
Using the sharper estimate \eqref{e:sharper} and imposing $a \geq C r^2$, so to guarantee separation of the two ends of $\Gamma_{a,r}$ where we are changing the metric, a more careful construction of the metric following the approach of \cite{DHS} allows the error estimate 
\begin{equation}\label{e:improved}
\|g_{a,r}-e\|_{C^k} \leq C_k a^{m-1}r^{2-m-k}\, .
\end{equation}
\end{remark}

\subsection{Proof of Proposition \ref{p:exact}} 
We use the weighted catenoid argument of \cite[Section~2]{KP}. While there are some technical differences, the blueprint of the argument is taken from the latter reference. Starting from a small catenoid truncated at horizontal radius \(r\),
we deform its two ends by adding \(w\) and \(-w\), respectively.
We then correct the mean curvature by a further small deformation,
allowing the boundary height to vary. The correction is obtained by
a contraction argument. Its linear ingredient is
Lemma~\ref{l:linear}, whose proof is given in
Appendix~\ref{app:linear}.

Evenness gives \(Dw(0)=0\), and hence \(w(x)=O(|x|^2)\) and
\(Dw(x)=O(|x|)\) near the origin. We fix \(\alpha=1/2\).
Unless indicated otherwise, constants may depend on \(m,w\) and on
the derivative order, but are independent of \(a,r\).

\smallskip
\noindent\emph{The first immersion.}
We parametrize the ``unit waist" \(m\)-dimensional catenoid over $\mathbb R \times \mathbb S^{m-1} \ni (t, \omega)$ by
\[
 X(t,\omega)=(\rho(t)\omega,Z(t)),\qquad
 \rho(t)=\cosh((m-1)t)^{1/(m-1)},\qquad
 Z(t)=\int_0^t\rho(s)^{2-m}\,ds.
\]
Its induced metric and unit normal are
\[
 g_C=\rho^2(dt^2+g_{\mathbb S^{m-1}}),\qquad
 N=(-\rho^{1-m}\omega,\vartheta),\qquad
 \vartheta=\rho'/\rho=\tanh((m-1)t).
\]
The upper end is the radial graph
\[
 Z_{\rm end}(R)=\int_1^R\frac{ds}{\sqrt{s^{2m-2}-1}},
 \qquad Z(t)=Z_{\rm end}(\rho(t))\quad(t\geq0).
\]
We set
\[
 \sigma=\frac a8,\qquad L=\frac r\sigma,\qquad
 D_L=[-T_L,T_L]\times\mathbb S^{m-1},\qquad \rho(T_L)=L.
\]
Thus \(\sigma X\) has neck radius \(\sigma\) and is truncated at
horizontal radius \(r\). We assume \(a_0(r)\) is small enough so that \(L\)
is larger than a fixed constant. The factor \(1/8\) will ensure the
precise graphical region in (d).

To match the prescribed graphs, we use a deformation field that is
normal near the waist and vertical near the boundary. We choose a
smooth cutoff \(\chi(\rho)\), equal to one for \(\rho\leq2\)
and zero for \(\rho\geq3\), with \(0\leq\chi\leq1\), and set
\[
 Y=\chi N+(1-\chi)\operatorname{sgn}(t)e_{m+1}\, .
 \]
 $Y$ is smooth and transversal to the catenoid and, for later use, it is convenient to quantify this transversality with the
 quantity
 \[
 \beta :=N\cdot Y=\chi+(1-\chi)|\vartheta|\geq c>0.
\]
We define
\begin{equation}\label{eq:firstimmersion}
 q(t,\omega)=\sigma^{-1}w(\sigma\rho(t)\omega),
 \qquad G=\sigma(X+qY).
\end{equation}
On a radius-\(R\) annulus, every fixed scaled derivative of \(q/R\)
is bounded by \(C_k\sigma R\leq C_kr\); on the fixed central region,
the corresponding bounds for \(q\) are \(C_k\sigma\).
Consequently \(G\) is an immersion for small \(r\).
On the ends it is the union of the graphs
\[
 \pm\bigl(\sigma Z_{\rm end}(|x|/\sigma)+w(x)\bigr).
\]
In particular, with \(H_0=\sigma Z(T_L)\), its boundary values are
\[
 G(\pm T_L,\omega)=
       \bigl(r\omega,\pm(H_0+w(r\omega))\bigr).
\]
We have therefore obtained the prescribed boundary shape. However the mean curvature of this model surface need not be $0$. We next
further deform the surface: the aim is to correct its mean curvature while retaining the boundary
data.

\smallskip
\noindent\emph{The equation for the correction.}
We seek a function \(z\) on \(D_L\) such that the corrected immersion
\[
 G_z=G+\sigma zY=\sigma\bigl(X+(q+z)Y\bigr)
\]
is minimal. Let \(\mathcal M(v)\) be the scalar mean curvature of \(X+vY\),
pulled back to \(D_L\). We use the convention for which the Jacobi
operator is \(\Delta_C+|A_C|^2\). This map is defined for
deformations small in scaled \(C^{2,\alpha}\) norms relative to the
local radius; the bounds on \(q/R\) place \(q\) in this neighborhood.
Scaling by \(\sigma\) does not affect minimality.

To preserve the boundary shape, we require \(z\) to equal the same
unknown constant \(\lambda\) on both boundary components. We also
require \(z\) to be even in \(t\), invariant under each horizontal
coordinate reflection, and to have zero spherical mean at the
waist. The problem to solve is therefore
\begin{equation}\label{eq:nonlinearproblem}
 \begin{cases}
  \mathcal M(q+z)=0 &\text{on }D_L,\\
  z|_{t=\pm T_L}=\lambda,\\
  \displaystyle\int_{\mathbb S^{m-1}}z (0,\omega)\,d\omega=0.
 \end{cases}
\end{equation}
Indeed, its boundary condition gives
\begin{equation}\label{eq:boundaryheight}
 G_z(\pm T_L,\omega)=
   \bigl(r\omega,\pm(H+w(r\omega))\bigr),\qquad
 H=H_0+\sigma\lambda.
\end{equation}
Thus \(\lambda\) determines the adjustment of the height.
The waist condition fixes the remaining radial freedom in the
correction; it does not prescribe the boundary height.

We linearize at the catenoid. Since the normal component of \(vY\)
is \(\beta v\), the resulting operator is
\begin{equation}\label{eq:jacobi}
 \mathcal Lv=D\mathcal M(0)[v]=J(\beta v),\qquad
 J=\rho^{-2}\left(\partial_t^2+(m-2)\vartheta\partial_t
       +\Delta_{\mathbb S^{m-1}}+m(m-1)\rho^{2-2m}\right).
\end{equation}
Writing
\[
 \mathcal R(v)=\mathcal M(q+v)-\mathcal M(q)-\mathcal Lv,
\]
we can express the equation as
\begin{equation}\label{eq:split}
 \mathcal Lz=-\mathcal M(q)-\mathcal R(z).
\end{equation}
Here \(\mathcal M(q)\) is the error of the first approximation.
The remainder satisfies \(\mathcal R(0)=0\); it includes the small
linear term \((D\mathcal M(q)-\mathcal L)z\), as well as the nonlinear
terms.

\smallskip

\noindent\emph{The linear lemma.}
We use norms adapted to the radius of the catenoid. For a function
\(v\) on \(D_L\), we set
\[
 \norm{v}_{k,\alpha;\delta}
 =\sup_{1\leq R\leq L}R^{-\delta}
   \norm{v}_{C^{k,\alpha}(A_R,R^{-2}g_C)},\qquad
 A_R=D_L\cap\{R/2<\rho<2R\}.
\]
Fixed smooth charts cover the waist, and the norms include boundary
half-charts. We denote by \(C^{k,\alpha}_\delta(D_L)\) the space
with this norm restricted to functions even in \(t\) and invariant
under each horizontal coordinate reflection. These symmetries are
preserved by all the operators above.

Let \(\mathcal E_L\) be the closed subspace of
\(C^{2,\alpha}_{1/2}(D_L)\times\R\) consisting of pairs \((z,\lambda)\)
satisfying the last two conditions in \eqref{eq:nonlinearproblem}.
We set \(\mathcal F_L=C^{0,\alpha}_{-3/2}(D_L)\) and equip
\(\mathcal E_L\) with the norm
\[
 \norm{(z,\lambda)}_{\mathcal E_L}
       =\norm{z}_{2,\alpha;1/2}+L^{-1/2}|\lambda|.
\]
These are Banach spaces. The two weights match the order of
\(\mathcal L\): a function controlled at scale \(R^{1/2}\) has
second derivatives controlled at scale \(R^{-3/2}\).
The boundary constant is measured at the outer scale \(L\).

\begin{lemma}[Uniform linear estimate]\label{l:linear}
There are \(L_0>3\) and \(C_{\rm lin}>0\), depending only on \(m\)
and the fixed cutoff, such that for every \(L\geq L_0\) and
\(f\in\mathcal F_L\) there is a unique pair
\((z,\lambda)\in\mathcal E_L\) satisfying
\begin{equation}\label{eq:linear}
 \mathcal Lz=f.
\end{equation}
Moreover,
\begin{equation}\label{eq:inverse}
 \norm{(z,\lambda)}_{\mathcal E_L}
       \leq C_{\rm lin}\norm{f}_{0,\alpha;-3/2}.
\end{equation}
\end{lemma}

The lemma is proved in Appendix~\ref{app:linear}. We denote its
solution operator by \(\mathcal P_Lf=(z,\lambda)\).
We can now solve \eqref{eq:split} once we estimate its error and
remainder.

\smallskip
\noindent\emph{The nonlinear estimates.}
We will show that
\begin{equation}\label{eq:residualnorm}
 \norm{\mathcal M(q)}_{0,\alpha;-3/2}\leq C_{\rm err}\sigma
\end{equation}
and that, for \(\norm{z}_{2,\alpha;1/2},\norm{v}_{2,\alpha;1/2}\leq M\),
\begin{equation}\label{eq:nonlinear}
 \norm{\mathcal R(z)-\mathcal R(v)}_{0,\alpha;-3/2}
 \leq C_{\rm rem}(r+M)\norm{z-v}_{2,\alpha;1/2}.
\end{equation}
Here \(r\) and \(M\) are smaller than fixed constants.

For the first estimate, consider an upper vertical end of \(G/\sigma\).
Its graph is
\[
 Z_{\rm end}(|\xi|)+\sigma^{-1}w(\sigma\xi).
\]
We set \(p=DZ_{\rm end}(|\xi|)\), \(p_0=Dw(\sigma\xi)\), and
\(\Phi(p)=p/\sqrt{1+|p|^2}\). Both summands are minimal graphs.
The residual is therefore the divergence of
\[
 \Phi(p+p_0)-\Phi(p)-\Phi(p_0).
\]
Since \(D^2\Phi(0)=0\), the integral Taylor formula bounds this expression
by \(C(|p_0||p|^2+|p_0|^2|p|)\), with the corresponding scaled
H\"older estimates. On the radius-\(R\) end annulus,
\(p_0=O(\sigma R)\) and \(p=O(R^{1-m})\), including every fixed
scaled derivative. Consequently, for every fixed \(k\geq0\),
\begin{equation}\label{eq:residualpoint}
 |D^k\mathcal M(q)|\leq
 C_k\left(\sigma R^{2-2m-k}+\sigma^2R^{2-m-k}\right).
\end{equation}
Here the derivatives are in the horizontal coordinates on the end.
On the central region the residual is \(O(\sigma)\) in every fixed
norm. It follows that
\[
 \norm{\mathcal M(q)}_{0,\alpha;-3/2}
 \leq C\left(\sigma+\sigma^2\max\{1,L^{7/2-m}\}\right)
 \leq C\sigma.
\]
Only \(m=3\) gives a positive exponent in the maximum, and then
\(\sigma^2L^{1/2}=\sigma^{3/2}\sqrt r\leq C\sigma\).
This proves \eqref{eq:residualnorm}. The cancellation uses the
minimality of the full function \(w\).

For \eqref{eq:nonlinear}, we divide a radius-\(R\) patch by \(R\).
On this patch the reference
and correction displacements have norms \(O(\sigma R)\) and
\(O(MR^{-1/2})\). The differential of the mean-curvature map differs
from its value at the undeformed catenoid by
\(O(\sigma R+MR^{-1/2})\). A difference of corrections contributes
the factor \(R^{-1/2}\), mean curvature scales by \(R^{-1}\),
and the source weight is \(R^{3/2}\). These factors give
\eqref{eq:nonlinear}.

\smallskip
\noindent\emph{The contraction and regularity.}
We apply the linear lemma on the closed ball
\[
 \mathcal B_{L,\sigma}
   =\{(z,\lambda)\in\mathcal E_L:
                    \norm{(z,\lambda)}_{\mathcal E_L}\leq A\sigma\},
 \qquad A\geq2C_{\rm lin}C_{\rm err}.
\]
On this ball we define
\begin{equation}\label{eq:fixedpointmap}
 \mathcal T(z,\lambda)
      =\mathcal P_L\bigl[-\mathcal M(q)-\mathcal R(z)\bigr].
\end{equation}
Its values automatically satisfy the boundary and waist conditions.
Moreover, \(\norm{\mathcal T(0,0)}_{\mathcal E_L}\leq A\sigma/2\),
and its Lipschitz constant is at most
\(C_{\rm lin}C_{\rm rem}(r+A\sigma)\).
We choose \(r_0\) small enough that this constant is at most
\(1/2\), using \(\sigma<r\). Thus \(\mathcal T\) maps
\(\mathcal B_{L,\sigma}\) into itself and is a contraction.
Its fixed point solves \eqref{eq:nonlinearproblem} and satisfies
\begin{equation}\label{eq:solution}
 \norm{z}_{2,\alpha;1/2}+L^{-1/2}|\lambda|\leq C\sigma.
\end{equation}

The solution is smooth up to the boundary, and the estimate extends
to every derivative order:
\begin{equation}\label{eq:higher}
 \norm{z}_{k,\alpha;1/2}\leq C_k\sigma.
\end{equation}
To see that the same weight is retained, set \(U_R=z/R\) and
\(Q_R=q/R\) on a rescaled patch. Subtracting the equation at \(Q_R\)
from the equation at \(Q_R+U_R\) gives a uniformly elliptic
quasilinear equation for \(U_R\).
Its inhomogeneous term is the rescaled reference residual, and
every fixed derivative of that term is bounded by
\(C_k\sigma R^{-1/2}\), by \eqref{eq:residualpoint}.
All remaining terms vanish at \(U_R=0\). Starting from
\eqref{eq:solution}, Schauder induction on nested scaled patches
therefore preserves the bound \(C_k\sigma R^{-1/2}\) for \(U_R\).
At the outer boundary its trace is the constant \(\lambda/R\),
with the same bound, so boundary estimates apply as well.
The higher constants may depend on higher derivatives of \(w\), but
no new smallness threshold depending on \(k\) is needed.

\smallskip
\noindent\emph{The immersion and the geometric estimates.}
For the solution just obtained, we set
\[
 F_0=G_z=G+\sigma z Y=\sigma\bigl(X+(q+z)Y\bigr)
 \quad\text{on }D_L.
\]
On the fixed central region this is a small smooth perturbation
of a catenoid, and on the ends it is a vertical graph.
The relative \(C^1\) bounds ensure that it is a smooth minimal immersion.
We define
\[
 \overline F_{a,r}(\omega,s)=F_0(T_Ls,\omega),
 \qquad (\omega,s)\in\mathbb S^{m-1}\times[-1,1]\, .
\]
Taking \(\Sigma_{a,r}\) to be the image of the map proves (b).

The horizontal coordinate of \(F_0/\sigma\) is
\[
 \bigl(\rho-\chi(q+z)\rho^{1-m}\bigr)\omega.
\]
Where \(\chi\ne0\), it has radius at most \(3+C\sigma<4\)
after decreasing the threshold for \(a\). Thus that entire part of
the image has horizontal radius less than \(4\sigma=a/2\).
On the vertical ends the horizontal coordinate is exactly
\(\sigma\rho\omega\). Consequently the part with \(a<|x|<r\)
consists precisely of two graphs. By reflection symmetry they are
\(\pm\zeta\), where
\begin{equation}\label{eq:physical}
 \zeta(x)=\sigma Z_{\rm end}(|x|/\sigma)
                      +w(x)+\sigma z_+(x/\sigma).
\end{equation}
Here \(z_+\) denotes the correction expressed in horizontal
coordinates on the upper end. Recall from \eqref{eq:boundaryheight}
that
\[
 H=H_0+\sigma\lambda=\sigma Z(T_L)+\sigma\lambda.
\]
The boundary condition in \eqref{eq:nonlinearproblem} gives
\[
 \overline F_{a,r}(\omega,\pm1)
       =\bigl(r\omega,\pm(H+w(r\omega))\bigr),
\]
which proves (c). The graphical description above proves (d).

To estimate the height, we use \(m\geq3\):
\[
 c_m:=Z(+\infty)=\int_1^\infty
                  \frac{ds}{\sqrt{s^{2m-2}-1}}<\infty,\qquad
 c_m-Z_{\rm end}(L)=O(L^{2-m}).
\]
Moreover \(|\lambda|\leq C\sigma L^{1/2}=C\sqrt{\sigma r}\).
It follows that
\begin{equation}\label{eq:height}
 H=c_m\sigma+
 O\left(\sigma^{m-1}r^{2-m}+\sigma^{3/2}r^{1/2}\right).
\end{equation}
In particular \(H/\sigma\to c_m>0\) at fixed \(r\), and
\(0<H\leq C a\) for sufficiently small \(a\).
We take \(a_0(r)\leq r^2\), and also \(a_0(r)<r/2\), which proves (a).

For \(r/2<|x|<r\), \eqref{eq:higher} implies
\[
 \left|D^k\bigl(\sigma z_+(x/\sigma)\bigr)\right|
       \leq C_k\sigma^{3/2}r^{1/2-k}.
\]
For \(k=0\), subtracting \(\sigma\lambda\) obeys the same bound.
The explicit catenoid profile similarly gives, for every \(k\geq0\),
\[
 \left|D^k\left(\sigma Z_{\rm end}(|x|/\sigma)
                         -\sigma Z(T_L)\right)\right|
       \leq C_k\sigma^{m-1}r^{2-m-k}.
\]
Substitution in \eqref{eq:physical}, followed by \(\sigma=a/8\),
proves \eqref{eq:target} and hence (e).

Finally, the relative \(C^1\) bounds used to construct \(F_0\) also give
\[
 F_0^*g_{\rm Euc}\leq C\sigma^2g_C.
\]
Indeed, on each radius-\(R\) patch the scaled displacement and its first
derivatives are bounded by
\(C(\sigma R+\sigma R^{-1/2})\leq Cr\).
The area of the compact annulus, counting multiplicity, is therefore
at most
\[
 C\sigma^m\int_{-T_L}^{T_L}\rho(t)^m\,dt
 \leq C\sigma^mL^m=Cr^m.
\]
The middle inequality follows from boundedness of \(\rho\) on
\([0,1]\) and a change of variables from \(t\) to \(\rho\) on
\([1,T_L]\), where \(\rho'/\rho\geq c>0\).
Reparametrization preserves area, and removing
the boundary does not change its \(m\)-dimensional measure. Hence
\[
 \int_{\mathbb S^{m-1}\times(-1,1)}
                d\operatorname{vol}_{F_{a,r}^*g_{\rm Euc}}
 \leq Cr^m,
\]
which implies (f) and completes the proof.

\begin{remark}
The choice \(\sigma=a/8\) puts the entire transition from normal
to vertical deformation inside \(\{|x|<a\}\).
If \(a\) is used as the neck scale, the proof gives the
graphical region \(Ka<|x|<r\) for a fixed \(K>1\).
No hypothesis on \(D^2w(0)\) is needed.
\end{remark}

\begin{remark}\label{r:embedded-3}
We explain here how to achieve an embedded model, for which the sharper estimate \eqref{e:sharper}. We cannot use the strategy above, since it does not give good enough bounds when $H$ falls in the range $(C r^2, \gamma r)$. The idea is then to construct $\Sigma_{a,r}$ in the following more sophisticated way. We split the classical catenoid of waist $a$ in two regions: the cylindrical region contained in $B_{C_0 a}\times \mathbb R$ and the annular region $(B_r \setminus B_{C_0 a})\times \mathbb R$. In the second, outer, region the catenoid consists of two graphs. We then solve the graphical minimal surface equation, imposing the boundary conditions $H(a,r) +w$ and $- H(a,r) -w$ on the outermost boundaries, and some yet unspecified boundary data $z$ and $-z$ at the innermost boundary. We then perturb the inner region of the catenoid to find a connected minimal surface which matches the two boundary data $-z$ and $z$. The union of the three surfaces is now a single embedded Lipschitz surface, but its tangents will be, most likely, discontinuous across the four interfaces. However, if we could choose the boundary data $z$ so that the tangents are continuous, then we have achieved a single smooth minimal surface. It is possible to find $z$ by setting up a classical implicit function theorem, using the symmetries and degrees of freedom of the problem to compensate the kernel of a suitable linearized operator. The improved estimate \eqref{e:sharper} follows then from classical estimates on minimal graphs.  
\end{remark}

\subsection{Proof of Proposition \ref{p:glued}} 
As for the proof of Proposition \ref{p:base-model} the idea is to apply Lemma \ref{lem:metric-formula} as in the proof of Proposition \ref{p:base-model}. We therefore first wish to find the matrix $A$ with the property that ${\rm div}\, (A\nabla\xi) = 0$ on $B_r\setminus B_{r/2}$, and $A$ is  $(1+|\nabla \xi|^2)^{-1/2}$ times the identity in a neighborhood of $\partial B_{r/2}$ and on $\{|\bar x|\geq r\}$. Then Lemma \ref{lem:metric-formula} will deliver a metric $g$ which makes the graphs of $\xi$ and $-\xi$ both minimal and which is the Euclidean metric in a neighborhood of $\partial B_r \times \mathbb R$ and $\{|\bar x|\geq r\} \times \mathbb R$. We continue the metric outside as the Euclidean metric. As for the estimates, we claim that, after setting $d = (1+ |\nabla \xi|^2)^{1/2}$, $A$ can be chosen so that
\begin{equation}\label{e:stimella}
\|A-d^{-1} {\rm Id}\|_{C^k}\leq C_k \|\xi - (w+H)\|_{C^{k+2}}\, .
\end{equation}
It is easy to see that 
\[
\|\xi-(w+H)\|_{C^k}\leq C_k\left(a^{m-1}r^{2-m-k}
                   +a^{3/2}r^{1/2-k}\right).
\]
We can therefore ensure $\|\nabla \xi\|_{C^0} \leq 1$ simply by restricting first the size of $W'$ so that $\|\nabla w\|_{C^0} \leq \frac{1}{2}$ and then restricting the size of $r$ and of $a\leq a_0 (r)$. Hence the explicit formula of Lemma \ref{lem:metric-formula} and \eqref{e:stimella} guarantee the desired estimates on $g_{a,r}-e$. 

We now come to the explicit construction of $A$. First of all, we can extend $\xi$ to the interior by setting it equal to $w+H$ in $B_{r/4}$ and using a partition of unity (with a radial cutoff). This will anyway retain the estimates on $\xi-(w+H)$. Next we compute  
\[
E:= - {\rm div} (d (x)^{-1} \nabla \xi (x))\, .
\]
$E$ vanishes identically in a neighborhood of $\partial B_{r/2}$. Hence, if we define $\tilde{E}$ as identically $0$ on $B_{r/2}$ and as identically $E$ outside, $\tilde{E}$ keeps the smoothness of $E$. In fact, because $w$ solves the Euclidean minimal surface equation, we also easily see that 
\[
\|\tilde{E}\|_{C^k} \leq C_k \|\xi - (w+H)\|_{C^{k+2}}\, .
\]
As in the proof of Proposition \ref{p:base-model}, we know that $D^2 w (0)$ is diagonal and with at least one nonzero entry, which upon relabeling the coordinates we can assume is the first one. We make the ansatz $A= d^{-1} {\rm Id} + b e_1\otimes e_1$ and then define
\[
b (x) = (\partial_1 \xi) (x)^{-1} \int_0^{x_1} \tilde{E} (t,x_2, \bar{x})\, dt \, 
\]
for $x_1\neq 0$, and extend it smoothly to $x_1=0$ using that $\partial^2_{11} \xi (0, x_2, \bar{x})$ does not vanish and De L'H\^ospital's Theorem. Note that $b$ is identically equal to $0$ for $x$ in a neighborhood of $B_{r/2}$ and when $|\bar{x}|\geq r$ (in both cases, the integration segment in the definition of $b$ lies all in the region where $\tilde{E}$ vanishes). The estimates for $b$ are straightforward. 

Clearly, the construction above does not guarantee that ${\rm div} (A\nabla \xi) = 0$ in $B_{r/2}$, but on the other hand this is not required.

\section{Proof of Theorem \ref{thm:main}}

We first fix $w$ as in Proposition \ref{p:base-model} and vanishing at the origin. Further we select $\bar p_k = (k^{-1}, 0, \ldots, 0)$, $p_k = (0,0, \bar{p}_k)$, and $r_k = 2^{-k}$, for $k\in \mathbb N\setminus \{0\}$. For $k$ sufficiently large, Proposition \ref{p:glued} will apply with $r=r_k$. Next choose $a_k$ small enough (depending on $r$) so that the conclusions of Proposition \ref{p:exact} and \ref{p:glued} apply after we translate the points $\bar{p}_k$ to the origin: we denote by $\Gamma_{a_k,r_k}$ the corresponding surfaces and by $g_k$ the corresponding metrics. We will also impose the inequality $a_k \leq 2^{-2^k}$. In particular, let $W'$, $h_k = H (a_k, r_k)$ and $\Gamma_{a_k, r_k}$ be as in Proposition \ref{p:glued} and we record the fact that $h_k \leq C 2^{-k}$. 

Assuming $k$ is large enough, that the $m-2$-dimensional balls
\[
S_k :=\{x: |\bar{x}-\bar{p}_k|\leq r_k\} \subset \mathbb R^{m-2}
\]
are pairwise disjoint. We start by defining the surface $\Sigma$ of Theorem \ref{thm:main} on the union of the slabs $W' \times S_k \times \mathbb R$ and this is done by setting it equal to the surface $p_k + \Gamma_{a_k, r_k}$ inside $B_{r_k} (p_k)\times \mathbb R$ and equal to the the union of the two graphs $\pm (h_k + w)$ on the remaining part of each slab, namely $((W'\times S_k) \setminus (B_{r_k} (p_k)) \times \mathbb R$. Notice that in a neighborhood of the boundary of each slab $W' \times S_k \times \mathbb R$ the surface coincides with the graphs of $\pm (h_k + w)$. We now need a smooth function as in the following lemma.

\begin{lemma}\label{lem:smooth_plateaus}
Assuming $r_k$, $\bar{p}_k$, and $h_k$ are as above, then there exists an integer \(K\) and a function \(f\in C^\infty(\R^{m-2})\)
such that
\[
 f(0)=0,\qquad f(x)>0\quad\text{for }x\ne0,\qquad
 f\equiv h_k\quad\text{on }B_{r_k}(\bar p_k)
 \quad\text{for every }k\geq K.
\]
Moreover, \(f\) can be chosen so that all its derivatives vanish at the origin.
\end{lemma}

Assuming the lemma, we then define the surface $\Sigma$ on the remaining portion of $W'\times \mathbb R^{m-2}$ by setting it equal to the union of the graphs $\pm (w+f)$. We are then in a position of applying Proposition \ref{p:base-model} and Proposition \ref{p:glued}. In particular we assume $W'$ and $\bar{W}'$ are small enough so that Proposition \ref{p:base-model} applies, and we get a correspondingly small neighborhood of the origin $\bar{W}'\subset \mathbb R^{m-2}$ and a metric $g_0$. We then set the final metric $g$ to be identically equal to $g_0$ in $W'\times \bar{W}'\times \mathbb R$ outside the slabs and equal to $g_k (x) := g_{a_k, r_k} (x-p_k)$ in each slab $W'\times (\bar{W}' \cap S_k)\times \mathbb R$. Since both $g_0$ and $g_k$ coincide with the Euclidean metric in a neighborhood of the portion of the boundary of $W'\times S_k \times \mathbb R$ which lies inside $W'\times \bar{W}'\times \mathbb R$, the metric is obviously smooth outside $W'\times \{\bar x = 0\}$. The fact that it is also smooth across $W'\times \{(\bar x =0\}$ follows easily from the estimates in Proposition \ref{p:glued} and the choice of the parameters (recall that $a_k \leq 2^{-2^k}$ and $r_k = 2^{-k}$, so the rapid decay of $a_k$ compared to $r_k$ offsets any negative power of $r_k$ appearing in the estimates). 

The other two claims of the theorem, namely the stationarity of the integral varifold associated with $\{0\}\cup \Sigma$ in the Riemannian manifold $(W'\times \bar{W}' \times \R,g)$ and the fact that the varifold tangent at $0$ is the horizontal plane with multiplicity $2$, are direct consequences of the construction and details are left to the reader. It only remains to show the existence of the function $f$, namely to prove Lemma \ref{lem:smooth_plateaus}, which is fairly straightforward: we include a sketch for the reader's convenience.  

\begin{proof}[Proof of Lemma \ref{lem:smooth_plateaus}]
We construct the transitions on the scale of the separation between
the centers, which is comparable to \(k^{-2}\).
Set
\[
 \delta_k=\frac{1}{10k^2}.
\]
The balls \(B_{\delta_k}(\bar p_k)\) are pairwise disjoint, because
their projections onto the first coordinate axis are disjoint:
\[
 \delta_k+\delta_{k+1}
 <\frac1k-\frac1{k+1}.
\]
Since \(2^{-k}=o(k^{-2})\), we can choose \(K\) such that
\(r_k<\delta_k/2\) for every \(k\geq K\).
Let \(\chi\in C_c^\infty(B_1)\) satisfy \(0\leq\chi\leq1\) and
\(\chi\equiv1\) on \(B_{1/2}\), and define
\[
 \chi_k(x)=\chi\!\left(\frac{x-\bar p_k}{\delta_k}\right),
 \qquad
 g(x)=
 \begin{cases}
  e^{-1/|x|^2},&x\ne0,\\
  0,&x=0.
 \end{cases}
\]
We set
\[
 f(x)=g(x)+\sum_{k\geq K}\chi_k(x)\bigl(h_k-g(x)\bigr).
\]
The verification that the function satisfies all the requirements is straightforward and left to the reader.
\end{proof}

\begin{remark}\label{r:embedded-4}
As already pointed out earlier, since we are choosing $a_k \ll r_k^2$, the neck regions in the above constructions will not be embedded. If we were to impose $a_k \geq C r_k^2$, we would achieve embedded neck regions, but the construction would not yield a $C^\infty$ metric. However, using the sharper estimate \eqref{e:improved} and the Ansatz $a_k = r_k^{2-\varepsilon}$ for $C$ large enough, we would get 
\begin{equation}\label{e:improved-2}
\|g-e\|_{C^j} \leq C_j r_k^{m-j-(m-1) \varepsilon}\, 
\end{equation}
in each slab.
This would still recover a $C^{m-1, 1-(m-1) \varepsilon}$ metric, as claimed in the first version of \cite{DHS}. 
\end{remark}

\appendix
\section{The linear estimate}\label{app:linear}

We prove Lemma~\ref{l:linear}, retaining the notation and symmetry
assumptions introduced in the main proof. The main point is to obtain
a bound independent of the truncation radius \(L\), since in the
geometric construction the ratio \(L=r/\sigma\) becomes arbitrarily
large. We first solve the equation on each fixed cylinder by separating
the spherical mean from the remaining angular components. The
normalization at the waist determines the former, while the reflection
symmetries give a coercive equation for the latter. We then prove
uniformity by a compactness argument, in which the choice of the
weight and the fact that the boundary trace is constant both play
a role. The arguments are standard in the literature and we just report them for the reader's convenience. 
Two examples are \cite[Proposition 4.3]{FP} and \cite[Section 3.3, Proposition 3]{Pacard}: for instance 
the contradiction argument used here to get the desired estimates in Step 3 is inspired by analogous arguments in
the latter references, the main difference being just on the choice of weights and the boundary data for the problem.
The proof uses also several standard facts in the theory of elliptic PDEs, such as interior and boundary Schauder estimates, Lax Milgram and interior derivative estimates for harmonic functions; a classical reference for these tools is \cite{GT}

\begin{proof}[Proof of Lemma~\ref{l:linear}]
Since \(\mathcal Lz=J(\beta z)\), it is convenient to work with the
normal component of the deformation, namely
\[
 v=\beta z,\qquad b=\beta(T_L)\lambda.
\]
The functions \(\beta\) and \(\beta^{-1}\), together with their scaled
derivatives, are uniformly bounded, so multiplication by either of
them preserves the weighted spaces with constants independent of \(L\).
This change of unknown also preserves all the conditions that accompany
the equation: \(\beta\) has the required reflection symmetries, equals
one at the waist, and takes the same constant value on the two outer
boundary components. Thus the problem becomes
\begin{equation}\label{eq:app-jacobi}
 Jv=f,\qquad v|_{t=\pm T_L}=b,\qquad
 \int_{\mathbb S^{m-1}}v(0,\omega)\,d\omega=0,
\end{equation}
with \(v\) even in \(t\) and invariant under each horizontal coordinate
reflection. We shall solve this problem and establish
\begin{equation}\label{eq:app-estimate}
 \norm{v}_{2,\alpha;1/2}+L^{-1/2}|b|
          \leq C\norm{f}_{0,\alpha;-3/2};
\end{equation}
the corresponding assertion for \((z,\lambda)\) will then follow by
multiplication by \(\beta^{-1}\).

\smallskip
\noindent\emph{1. Solvability on a fixed cylinder.}
We begin with \(L\) fixed, allowing the constants in this part of the
argument to depend on \(L\). Because the desired solution is even in
\(t\), we may solve on the half-cylinder
\(D_L^+=[0,T_L]\times\mathbb S^{m-1}\), where the symmetry is expressed
by the Neumann condition \(\partial_t v=0\) at \(t=0\).
The useful decomposition is
\[
 \bar v(t)=\frac{1}{|\mathbb S^{m-1}|}
              \int_{\mathbb S^{m-1}}v(t,\omega)\,d\omega,
 \qquad v^\perp=v-\bar v,
\]
and we write \(f=\bar f+f^\perp\) in the same way. Since the
coefficients of \(J\) depend only on \(t\), the operator commutes
with spherical averaging, so these two components can be solved
for separately.

For the spherical mean, the equation and the conditions at the waist
give the ordinary differential equation
\[
 \bar v''+(m-2)\vartheta\bar v'
       +m(m-1)\rho^{2-2m}\bar v=\rho^2\bar f,
 \qquad \bar v(0)=\bar v'(0)=0.
\]
The first initial condition is precisely the waist normalization,
whereas the second follows from evenness. They determine a unique
solution on \([0,T_L]\), whose value at \(T_L\) is then taken as the
boundary constant \(b\). This explains the role of the free boundary
value in the lemma: the radial component is determined by its data at
the waist, and its outer value is an output of the construction.

It remains to solve for the part with zero spherical mean. Since the
full boundary trace is constant, subtracting the radial component
leaves homogeneous Dirichlet data at the outer boundary, and therefore
\[
 Jv^\perp=f^\perp,\qquad
 \partial_t v^\perp|_{t=0}=0,\qquad v^\perp|_{t=T_L}=0.
\]
Here and below the zero-mean condition on a nonradial function is
imposed on every spherical slice. To obtain a coercive formulation of
this problem, we use the positive function
\[
 h=\rho^{1-m},\qquad v^\perp=hp,\qquad
 \mu_\ell=\ell(\ell+m-2).
\]
The numbers \(\mu_\ell\) are the eigenvalues of
\(-\Delta_{\mathbb S^{m-1}}\), with
\(\mu_1=m-1\) and \(\mu_2=2m\). The choice of \(h\) is suggested by
the Jacobi fields generated by horizontal translations: their normal
components are, up to sign, \(h\omega_i\), and hence
\(J(h\omega_i)=0\). Dividing by \(h\) replaces the radial Jacobi
potential by the constant \(\mu_1\). More explicitly, substitution
in the equation, using \(h'/h=-(m-1)\vartheta\), gives
\begin{equation}\label{eq:app-conjugated}
 (\rho^{-m}p_t)_t+
 \rho^{-m}(\Delta_{\mathbb S^{m-1}}p+\mu_1p)=\rho f^\perp,
 \qquad p_t|_{t=0}=0,\quad p|_{t=T_L}=0.
\end{equation}
The Neumann condition is preserved because \(h'(0)=0\).

The reflection symmetries now supply the needed positivity.
A spherical harmonic of degree one is a linear combination of the
coordinate functions \(\omega_i\), so a function invariant under
each coordinate reflection has no component of this degree.
Once its spherical mean is also removed, the first possible
eigenvalue is \(\mu_2\); consequently,
\[
 \int_{\mathbb S^{m-1}}|\nabla_\omega\psi|^2\,d\omega
       \geq\mu_2\int_{\mathbb S^{m-1}}|\psi|^2\,d\omega
\]
for every symmetric function \(\psi\) of zero mean. Applied on each
slice, this inequality controls the negative zeroth-order term in
the bilinear form
\[
 B(p,\psi)=\int_{D_L^+}\rho^{-m}
     \bigl(p_t\psi_t+\nabla_\omega p\cdot\nabla_\omega\psi
                                -\mu_1p\psi\bigr)\,dt\,d\omega,
\]
because
\[
 B(p,p)\geq\int_{D_L^+}\rho^{-m}
       \left(|p_t|^2+
          \left(1-\frac{\mu_1}{\mu_2}\right)|\nabla_\omega p|^2
       \right)\,dt\,d\omega.
\]
For fixed \(L\), the weight \(\rho^{-m}\) is bounded above and below
by positive constants, and the same angular inequality controls
the \(L^2\) norm. Thus \(B\) is coercive on the closed subspace of
\(H^1(D_L^+)\) consisting of symmetric functions with zero spherical
mean and zero trace at \(t=T_L\).

Lax--Milgram gives a unique solution in this space of the weak equation
\[
 B(p,\psi)=-\int_{D_L^+}\rho f^\perp\psi\,dt\,d\omega.
\]
Although the variational problem was formulated in the symmetry
subspace, this identity holds for every test function with zero outer
trace. Indeed, averaging a test function over the coordinate reflections
and subtracting its spherical mean leaves both sides unchanged.
We may therefore apply elliptic regularity for the mixed boundary
problem to obtain \(p\in C^{2,\alpha}\), with the natural Neumann
condition at \(t=0\) and the prescribed Dirichlet condition at \(T_L\).
The function \(v=\bar v+hp\), extended evenly across \(t=0\), solves
\eqref{eq:app-jacobi}. Uniqueness follows from the uniqueness of both
the radial initial-value problem and the variational problem.

The construction just given also indicates what remains to be proved.
Its coercivity constant can depend on \(L\), so it does not yet supply
the uniform bound required in the contraction argument. To obtain that
bound, we first identify the homogeneous solutions that could arise
as limits of a sequence violating it.

\smallskip
\noindent\emph{2. Jacobi fields on the complete catenoid.}
We claim that the only symmetric Jacobi field on the complete catenoid
satisfying
\begin{equation}\label{eq:app-growth}
 Jv=0,\qquad
 \int_{\mathbb S^{m-1}}v(0,\omega)\,d\omega=0,\qquad
 |v|\leq C\rho^{1/2}
\end{equation}
is the zero field. The growth condition is exactly the one allowed by
the weighted norm in the lemma, so this statement will rule out a
nonzero limit near the waist when the truncation radius tends to
infinity.

The radial component vanishes by the homogeneous initial-value
problem from Step~1. We examine the remaining components by expanding
\(v(t,\cdot)\) in a real spherical-harmonic basis. The symmetries
exclude degree one, so any surviving nonconstant component has degree
\(\ell\geq2\). If \(v_\ell(t)\) denotes the coefficient of one fixed
basis element and we write \(v_\ell=\rho^{1-m}p_\ell\), the same
conjugation as in \eqref{eq:app-conjugated} gives
\begin{equation}\label{eq:groundstate}
 (\rho^{-m}p_\ell')'
       -(\mu_\ell-\mu_1)\rho^{-m}p_\ell=0,
 \qquad p_\ell'(0)=0.
\end{equation}
A nonzero solution has \(p_\ell(0)\ne0\), since otherwise both of its
initial data would vanish; changing the sign of the coefficient, we
may therefore suppose that \(p_\ell(0)>0\).

The positive difference \(\mu_\ell-\mu_1\) forces such a solution to
grow. Indeed, integrating \eqref{eq:groundstate} from the waist yields
\[
 \rho(t)^{-m}p_\ell'(t)
   =(\mu_\ell-\mu_1)
       \int_0^t\rho(s)^{-m}p_\ell(s)\,ds.
\]
As long as \(p_\ell\) is positive, this identity makes it increasing,
so it can never reach zero. For \(t\geq1\), the integral over the
fixed interval \([0,1]\) is already a positive lower bound for the
right-hand side, and consequently
\(p_\ell'(t)\geq c\rho(t)^m\). Since \(\rho(t)\) is comparable to
\(e^t\) for large \(t\), a further integration gives
\(p_\ell(t)\geq c'\rho(t)^m\), and hence
\[
 v_\ell(t)=\rho(t)^{1-m}p_\ell(t)\geq c'\rho(t)
\]
for all sufficiently large \(t\). On the other hand, projecting the
bound in \eqref{eq:app-growth} onto the chosen spherical harmonic
gives \(|v_\ell(t)|\leq C_\ell\rho(t)^{1/2}\), a contradiction.
Every angular coefficient must therefore vanish along with the radial
component, which proves the claim.

\smallskip
\noindent\emph{3. The uniform weighted estimate.}
We now return to the solutions on finite cylinders. Local elliptic
estimates control their derivatives once their weighted supremum is
controlled, so we introduce
\[
 S_L(v)=\sup_{D_L}\rho^{-1/2}|v|
\]
and first record the estimate
\begin{equation}\label{eq:schauder}
 \norm{v}_{2,\alpha;1/2}
 \leq C\left(\norm{f}_{0,\alpha;-3/2}+S_L(v)\right),
\end{equation}
where \(C\) is independent of \(L\geq L_0\).
To see why the constant is uniform, one applies Schauder estimates
on annular patches after scaling each patch by its radius \(R\).
The catenoid has uniformly controlled geometry in these coordinates,
and the factor \(R^2\) contributed by the second-order equation
matches the difference between the two weights. Near the waist one
uses a fixed family of charts, while near the outer sphere one uses
boundary charts at scale \(L\). In the latter charts, the Dirichlet
trace is the constant \(b\), whose weighted norm is bounded by
\(S_L(v)\), so no additional boundary term is needed in
\eqref{eq:schauder}.

It remains to prove
\begin{equation}\label{eq:app-supremum}
 S_L(v)\leq C\norm{f}_{0,\alpha;-3/2}.
\end{equation}
The difficulty is that a sequence may have most of its weighted
amplitude farther and farther out along the catenoid. We therefore
argue by contradiction and center the compactness argument at points
where that amplitude is maximal. If \eqref{eq:app-supremum} failed,
linearity would allow us to normalize a sequence of solutions so that
\[
 Jv_j=f_j,\qquad
 S_{L_j}(v_j)=1,\qquad
 \norm{f_j}_{0,\alpha;-3/2}\longrightarrow0.
\]
Each \(v_j\) retains the symmetries and the zero waist average, and
its common boundary value will be denoted by \(b_j\).
We choose a maximizing point \((t_j,\omega_j)\), set
\(R_j=\rho(t_j)\), and use evenness to arrange \(t_j\geq0\).

After passing to a subsequence, we may assume \(L_j\to\infty\).
Indeed, if a subsequence converged to a finite truncation radius,
we could identify the cylinders with a fixed cylinder and use
\eqref{eq:schauder} to obtain a nonzero homogeneous solution of
\eqref{eq:app-jacobi}, contradicting the uniqueness proved in Step~1.
For an exhausting sequence there are two possible locations for the
maximizing points, and these give two different limiting equations.

\smallskip
\noindent\emph{Maximizing points at bounded radii.}
Suppose first that \(R_j\) remains bounded. Then the maximizing points
stay in a compact part of the catenoid, where
\eqref{eq:schauder} and a diagonal compactness argument give convergence
to a Jacobi field \(v_\infty\) on the complete catenoid.
The symmetries, the waist normalization, and the bound
\(|v_j|\leq\rho^{1/2}\) all pass to this limit, so \(v_\infty\)
satisfies \eqref{eq:app-growth}. After a further subsequence the
maximizing points themselves converge, and their normalization shows
that \(v_\infty\) is nonzero at the limiting point. This contradicts
Step~2 and rules out the first possibility.

\smallskip
\noindent\emph{Maximizing points whose radii tend to infinity.}
We now suppose that \(R_j\to\infty\). To keep the maximizing point
visible in the limit, we rescale both the horizontal coordinates and
the function values. Writing \(v_j^+(\xi)\) for \(v_j\) in horizontal
coordinates on the upper end, we set
\[
 \widetilde v_j(x)=R_j^{-1/2}v_j^+(R_jx),\qquad
 \Lambda_j=\frac{L_j}{R_j}.
\]
The rescaled functions are defined for
\(3/R_j<|x|<\Lambda_j\), with continuous extensions to the outer
boundary, and their normalization becomes
\[
 |\widetilde v_j(x)|\leq |x|^{1/2},\qquad
 |\widetilde v_j(\omega_j)|=1.
\]
In particular the maximizing points now lie on the unit sphere,
whereas the outer boundaries lie at radii
\(\Lambda_j\geq1\). Passing to a subsequence, we may assume
\(\Lambda_j\to\Lambda\in[1,\infty]\).

The limiting equation is Euclidean because, on every compact set
away from the origin, the rescaled catenoidal ends converge smoothly
to a plane and their rescaled Jacobi potentials tend to zero.
The forcing term also tends to zero: the coordinate rescaling
contributes \(R_j^2\) and the amplitude normalization contributes
\(R_j^{-1/2}\), so the rescaled right-hand side satisfies
\[
 \left|R_j^{3/2}f_j^+(R_jx)\right|
       \leq \norm{f_j}_{0,\alpha;-3/2}|x|^{-3/2}
       \longrightarrow0.
\]
Elliptic compactness therefore gives a harmonic limit \(v_\infty\)
on \(B_\Lambda\setminus\{0\}\) when \(\Lambda<\infty\), and on
\(\R^m\setminus\{0\}\) when \(\Lambda=\infty\).
It inherits the bound
\[
 |v_\infty(x)|\leq |x|^{1/2},
\]
which has two consequences: the singularity at the origin is removable,
and the value of the harmonic extension at the origin is zero.

When \(\Lambda<\infty\), the outer boundary remains part of the
limiting problem, so we must also keep track of its trace.
The normalized boundary constants satisfy
\[
 |R_j^{-1/2}b_j|\leq \Lambda_j^{1/2},
\]
and hence converge along a subsequence to a constant \(b_\infty\).
Boundary Schauder estimates, after identifying the outer spheres,
give convergence up to \(\partial B_\Lambda\) and show that
\(v_\infty=b_\infty\) there. The unique harmonic function on a ball
with constant boundary values is constant, and its value at the origin
therefore forces \(b_\infty=0\) and \(v_\infty=0\).

When \(\Lambda=\infty\), there is no outer boundary in the limiting
problem, but the growth bound gives the same conclusion.
For any fixed \(x_0\in\R^m\), the interior gradient estimate on
\(B_s(x_0)\) yields
\[
 |Dv_\infty(x_0)|
 \leq \frac{C}{s}\sup_{B_s(x_0)}|v_\infty|
 \leq C\frac{(|x_0|+s)^{1/2}}{s}\longrightarrow0
 \qquad\text{as }s\to\infty.
\]
Thus \(v_\infty\) is constant on \(\R^m\), and its value at zero
again implies that it vanishes.

In either case this conclusion contradicts the normalization at the
maximizing points. After choosing a convergent subsequence of
\(\omega_j\), convergence near the unit sphere gives a limiting value
of absolute size one. If \(\Lambda=1\), that sphere is the outer
boundary of the limiting domain, and it is precisely the boundary
convergence established above that preserves the nonzero value.
We have therefore excluded every possible location of the maximizing
points, proving \eqref{eq:app-supremum}.

\smallskip
\noindent\emph{4. The estimate for the original unknowns.}
Combining the weighted supremum bound with \eqref{eq:schauder} gives
the \(C^{2,\alpha}\) estimate in \eqref{eq:app-estimate}.
The boundary constant is controlled at the same time, since
\(L^{-1/2}|b|\leq S_L(v)\). Returning to
\[
 z=\beta^{-1}v,\qquad \lambda=\beta(T_L)^{-1}b,
\]
we obtain \eqref{eq:inverse} from the uniform multiplication bounds
noted at the start of the proof. Together with the existence and
uniqueness on each cylinder established in Step~1, this proves
Lemma~\ref{l:linear}.
\end{proof}

\bibliographystyle{plain}
\bibliography{references-3.bib}

\end{document}